\documentclass[10pt,a4paper]{amsart}

\theoremstyle{plain}
\newtheorem{theorem}{Theorem}[section]
\newtheorem{lemma}[theorem]{Lemma}

\newtheorem{corollary}[theorem]{Corollary}
\theoremstyle{definition}

\theoremstyle{remark}
\newtheorem{remark}[theorem]{Remark}

\newcommand{\bQ}{\mathbb{Q}}

\def\bin #1#2 {\left( \matrix { #1 \cr #2 \cr } \right) }

\begin{document}

\title[On a resolution of singularities with two strata]
{On  a resolution of singularities with two strata}

\author{Vincenzo Di Gennaro }
\address{Universit\`a di Roma \lq\lq Tor Vergata\rq\rq, Dipartimento di Matematica,
Via della Ricerca Scientifica, 00133 Roma, Italy.}
\email{digennar@axp.mat.uniroma2.it}

\author{Davide Franco }
\address{Universit\`a di Napoli
\lq\lq Federico II\rq\rq, Dipartimento di Matematica e
Applicazioni \lq\lq R. Caccioppoli\rq\rq, P.le Tecchio 80, 80125
Napoli, Italy.} \email{davide.franco@unina.it}

\abstract Let $X$ be a complex, irreducible, quasi-projective
variety, and $\pi:\widetilde X\to X$ a resolution of singularities
of $X$. Assume that the singular locus ${\text{Sing}}(X)$ of $X$ is
smooth, that the induced map $\pi^{-1}({\text{Sing}}(X))\to
{\text{Sing}}(X)$ is a smooth fibration admitting a cohomology
extension of the fiber, and that $\pi^{-1}({\text{Sing}}(X))$ has a
negative normal bundle in $\widetilde X$. We present a very short
and explicit proof of the Decomposition Theorem for $\pi$, providing
a way to compute the intersection cohomology of $X$ by means  of the
cohomology of $\widetilde X$ and of $\pi^{-1}({\text{Sing}}(X))$.
Our result applies to special Schubert varieties with two strata,
even if $\pi$ is non-small. And to certain hypersurfaces of $\mathbb
P^5$ with one-dimensional singular locus.



\endabstract
\maketitle

\bigskip
\section{Introduction}

The Decomposition Theorem is a beautiful and very deep result about
algebraic maps. In the words of MacPherson \lq\lq it contains as
special cases the deepest homological properties of algebraic maps
that we know\rq\rq  \cite{Mac83}, \cite{Williamson}. In literature,
one can find different approaches to the Decomposition Theorem
\cite{BBD}, \cite{DeCM2}, \cite{DeCM3}, \cite{Saito},
{\cite{Williamson}. Let us say they have in common a fairly heavy
formalism,  that may discourage the reader to the point that the
Decomposition Theorem is often used like a "black box" by many
authors \cite{Haines}, \cite{Billey}. Furthermore, it is often very
difficult to calculate the intersection cohomology of a singular
algebraic variety, unless either the singular locus if finite, or
the variety admits a small resolution with known Betti numbers.

However, there are many special cases for which the Decomposition
Theorem admits a simplified approach. One of these is  the case of
varieties with isolated singularities. This is a key point also in
the general case since, as observed in \cite[Remark
2.14]{Williamson}, the proof of the Decomposition Theorem proceeds
by induction on the dimension of the strata of the singular locus.

For instance, in our previous work \cite{DGF3}, we reduced the proof
of the Decomposition Theorem for varieties with isolated
singularities, to the vanishing of certain maps between ordinary
cohomology groups \cite[Theorem 3.1]{DGF3}. This in turn is related
with the existence of a \lq\lq natural Gysin morphism\rq\rq. By a
natural Gysin morphism we mean a {topological bivariant class}
\cite[p. 83]{FultonCF}, \cite{DeCM}, \cite{DGF1}: $$\theta \in
T^0(\widetilde
X\stackrel{\pi}{\rightarrow}X)=Hom_{D^b(X)}(R\pi_*\mathbb
Q_{\widetilde X} , \mathbb Q_X), $$ commuting with restrictions to
the smooth locus of $X$ (here  $\pi: \widetilde X\rightarrow X$ is a
resolution of singularities  of $X$ with {isolated singularities}).
In \cite[Theorem 1.2]{DGF3}, we gave a complete characterization of
morphisms like $\pi$ admitting a natural Gysin morphism, providing a
relationship between the {Decomposition Theorem} and Bivariant
Theory. In fact, $\pi$ admits a natural Gysin morphism if and only
if $X$ is a $\bQ$-{intersection cohomology manifold}, i.e.
$IC^{\bullet}_X\cong \bQ_X[n]$ in $D^b(X)$ ($IC^{\bullet}_X$ denotes
the {intersection cohomology complex} of $X$ \cite[p. 156]{Dimca2},
\cite{Massey}). In this case, there is a unique natural Gysin
morphism $\theta$, and it arises from the Decomposition Theorem.

Our aim in this work is to develop another case for which the
Decomposition Theorem admits a simplified approach. More precisely,
we assume $X$ to be a complex, irreducible, quasi-projective variety
of dimension $n\geq 1$, and $$\pi:\widetilde X\to X$$ a resolution
of singularities of $X$. Moreover, we assume that the singular locus
${\text{Sing}}(X)$ of $X$ is smooth, and that the induced fibre
square diagram:
$$
\begin{array}{ccccccc}
\pi^{-1}({\text{Sing}}(X))&\stackrel
{}{\hookrightarrow}&\widetilde X\\
\stackrel {}{}\downarrow & &\stackrel
{\pi}{}\downarrow \\
{\text{Sing}}(X)&\stackrel
{}{\hookrightarrow}&X\\
\end{array}
$$
is such that $\pi^{-1}({\text{Sing}}(X))\to {\text{Sing}}(X)$ is a
smooth fibration, with negative normal bundle, admitting a
cohomology extension of the fiber (see Notations, $(iii)$, below,
for a precise statement of our assumptions).

Our main result is a very short and explicit proof of the
Decomposition Theorem (compare with Theorem \ref{ExDecTh}),
providing a way to compute the intersection cohomology of $X$ by
means  of the cohomology of $\widetilde X$ and of
$\pi^{-1}({\text{Sing}}(X))$ (Corollary \ref{ExIC}). In the last two
sections, we apply our main result. First, to  {\it special Schubert
varieties with two strata}, for which it is known to exist both a
small and a non-small resolution. Comparing our computation of the
intersection cohomology by means of Corollary \ref{ExIC}, with the
one given in \cite{CGM}, we find some polynomial identities
apparently not known so far (Remark \ref{concl2}, $(ii)$). Next, we
compute the intersection cohomology of certain hypersurfaces of
$\mathbb P^5$ with one-dimensional singular locus. As far as we
know, Corollary \ref{new} is completely new.

\bigskip
\section{Notations}

$(i)$ All cohomology and intersection cohomology groups are with
$\mathbb Q$-coefficients.

\medskip
$(ii)$ Let $Y$ be a complex, possible reducible, quasi-projective
variety. We denote by $H^{\alpha}(Y)$ and $IH^{\alpha}(Y)$ its
cohomology and intersection cohomology groups ($\alpha\in\mathbb
Z$). Let $D^b(Y)$ be the bounded derived category of sheaves of
$\mathbb Q$-vector spaces on $Y$. Let $\mathcal F^{\bullet}\in
D^b(Y)$ be a complex of sheaves. We denote by $\mathcal
H^\alpha(\mathcal F^{\bullet})$ its cohomology sheaves, and by
$\mathbb H^\alpha(\mathcal F^{\bullet})$ its hypercohomology
groups. Let $IC^{\bullet}_{Y}$ be the intersection cohomology
complex of $Y$. If $Y$ is irreducible, we have
$IH^{\alpha}(Y)=\mathcal H^\alpha(IC^{\bullet}_{Y}[-\dim Y])$. If
$Y$ is irreducible and nonsingular, and $\mathbb Q_Y$ is the
constant sheaf $\mathbb Q$ on $Y$, we have $IC^{\bullet}_{Y}\cong
\mathbb Q_Y[\dim_{\mathbb C} Y]$.

\medskip
$(iii)$ Let $X$ be a complex, irreducible, quasi-projective variety
of dimension $n\geq 1$, and $$\pi:\widetilde X\to X$$ a resolution
of singularities of $X$. This means that $\pi$ is a projective,
surjective, birational morphism, such that $\widetilde X$ is
irreducible and nonsingular. Fix a closed, nonsingular subvariety
$\Delta$ of $X$, of  pure dimension $m$. Consider the induced fibre
square commutative diagram:
\begin{equation}\label{D1}
\begin{array}{ccccccc}
\widetilde \Delta&\stackrel
{\jmath}{\hookrightarrow}&\widetilde X\\
\stackrel {\rho}{}\downarrow & &\stackrel
{\pi}{}\downarrow \\
\Delta&\stackrel
{\imath}{\hookrightarrow}&X,\\
\end{array}
\end{equation}
where $\widetilde \Delta =\pi^{-1}(\Delta)$, $\imath$ and $\jmath$
are the inclusion maps, and $\rho$ the restriction of $\pi$. We make
the following assumptions $(a1)$, $(a2)$, and $(a3)$:

\medskip
$(a1)$ ${\text{Sing}}(X)\subseteq \Delta$, and the induced  map
$\pi^{-1}(X\backslash \Delta)\to X\backslash \Delta$ is an
isomorphism.

\medskip
$(a2)$ $\widetilde\Delta$ is nonsingular, of pure dimension $m+p$,
and the map $\rho:\widetilde \Delta\to \Delta$ is a smooth
fibration, with fiber say $G$, well-defined up to diffeomorphisms,
such that the restriction map $H^{\alpha}(\widetilde \Delta)\to
H^{\alpha}(G)$ is onto for all $\alpha\in\mathbb Z$.

\medskip
\noindent In view of previous assumptions, the fiber $G$ is a
projective variety, nonsingular, purely dimensional, of dimension
$p$. Let $N$ be the normal bundle of $\widetilde\Delta$ in
$\widetilde X$, and set $$q:=n-m-p$$ its rank. Let $\overline c\in
H^{2q}(\widetilde \Delta)$ be the top Chern class of $N$, and let
$c\in H^{2q}(G)$ be the restriction of $\overline c$ to $G$.

\medskip
$(a3)$  The map $H^{\alpha}(G)\stackrel {\cdot\, \cup\,
c}{\longrightarrow} H^{\alpha+2q}(G)$, determined by cup-product
with $c$, is onto for all integers $\alpha\geq p-q$.

\medskip
\begin{remark}\label{BG}  Combining the Universal Coefficient Theorem with
the Poincar\'e Duality Theorem, it follows  that condition $(a3)$ is
equivalent to require that the map $H^{\alpha}(G)\stackrel {\cdot\,
\cup\, c}{\longrightarrow} H^{\alpha+2q}(G)$ is injective for all
integers $\alpha\leq p-q$. Notice that if $ p-q<0$, then condition
$(a3)$ is satisfied. In fact, in this case, we have
$H^{\alpha+2q}(G)=0$ for all $\alpha\geq p-q$ (and $H^{\alpha}(G)=0$
for all $\alpha\leq p-q$). Moreover, in view of the Hard Lefschetz
Theorem for ample bundles, the condition $(a3)$ is satisfied when
the restriction of $N$ to $G$, either of $N^{\vee}$ to $G$, is ample
\cite[p. 69]{L}.
\end{remark}

\medskip
$(iv)$ For all $\alpha\in\mathbb Z$, set
$A^{\alpha}:=H^{\alpha}(G)$, and $a^{\alpha}:=\dim H^{\alpha}(G)$.
When $\alpha\leq p-q$, set
$B^{\alpha+2q}:=\Im(H^{\alpha}(G)\stackrel {\cdot\, \cup\,
c}{\longrightarrow} H^{\alpha+2q}(G))\cong H^{\alpha}(G)$. When
$\alpha\geq p-q$, choose a subspace $C^{\alpha}\subseteq
H^{\alpha}(G)$ such that $C^{\alpha}\cong H^{\alpha+2q}(G)$ via
$H^{\alpha}(G)\stackrel {\cdot\, \cup\, c}{\longrightarrow}
H^{\alpha+2q}(G)$. Observe that by the Universal Coefficient Theorem
and the Poincar\'e Duality Theorem, it follows that
\begin{equation}\label{PD}
A^{p-q-\alpha}\cong C^{p-q+\alpha}\quad{\text{for every $\alpha\geq
0$.}}
\end{equation}

\medskip
\begin{remark}\label{Ray}  By previous assumption $(a2)$,
there exists a {\it cohomology extension
$H^{*}(G)\stackrel{\theta}{\to} H^{*}(\widetilde \Delta)$ of the
fiber} \cite[p. 256-258]{Spanier}. By the Leray-Hirsch Theorem
\cite[p. 182 and p. 195]{Voisin}, \cite[Lemma 2.5 and proof]{DGF2},
it determines a decomposition in $D^b(\Delta)$:
\begin{equation}\label{LH}
\sum_{\alpha=0}^{2p} A^{\alpha}\otimes \mathbb Q_{\Delta}[-\alpha]
\stackrel{\theta}{\cong}R\rho_*\mathbb Q_{\widetilde\Delta}.
\end{equation}
\end{remark}

\medskip
$(v)$ We define the following complex $\mathcal F^{\bullet}$ in
$D^b(X)$. It will appear in the claim of Theorem \ref{ExDecTh}
below. When $p-q\geq 0$, we set:
$$
\mathcal F^{\bullet}:=\left(\sum_{\alpha=0}^{p-q}A^{\alpha}\otimes_{\mathbb
Q}
R{\imath}_*\mathbb Q_{\Delta}[n-2q-\alpha]\right)
\oplus
\left(\sum_{\alpha=p-q+1}^{2p-2q}C^{\alpha}\otimes_{\mathbb
Q}
R{\imath}_*\mathbb Q_{\Delta}[n-2q-\alpha]\right).
$$
When $p-q<0$, we set $\mathcal F^{\bullet}:=0$. If $p-q = 0$, then
we simply have
$$
\mathcal F^{\bullet}=H^0(G)\otimes_{\,\mathbb Q} R{\imath}_*\mathbb
Q_{\Delta}[n-2q].
$$

\medskip
$(vi)$ We denote by $IH_{ X}(t)$  the Poincar\'e polynomial of the
intersection cohomology of $X$, i.e.
$$
IH_{X}(t):=\sum_{\alpha\in\mathbb Z}\dim_{\mathbb
Q}IH^{\alpha}(X)\,t^{\alpha}.
$$
We denote by $H_{\widetilde X}(t)$ and $H_{\Delta}(t)$ the
Poincar\'e polynomials of the cohomology of $\widetilde X$ and
$\Delta$, i.e.
$$
H_{\widetilde X}(t):=\sum_{\alpha\in\mathbb Z}\dim_{\mathbb
Q}H^{\alpha}(\widetilde X)\,t^{\alpha},\quad{\text{and}}\quad
H_{\Delta}(t):=\sum_{\alpha\in\mathbb Z}\dim_{\mathbb
Q}H^{\alpha}(\Delta)\,t^{\alpha}.
$$

\medskip
$(vii)$ We define the polynomial  $g(t)$ as follows. First, when
$p-q\geq 0$, we set:
$$
r(t):=\frac{1}{2}a^{p-q}t^{p-q}+\sum_{0\leq \alpha\leq
p-q-1}a^{\alpha}t^{\alpha}.
$$
Next we define:
$$
g(t):=t^{2q}r(t)+t^{2p}r\left(t^{-1}\right).
$$
When $p-q<0$, we set $g(t):=0$. When $p-q=0$ and $G$ is connected,
we simply have $g(t)=t^{2p}$. We denote by $f(t)$ the Poincar\'e
polynomial of the complex $\mathcal F^{\bullet}[-n]$, i.e.
$$
f(t):=\sum_{\alpha\in {\mathbb Z}} {\mathbf h}^{\alpha}(\mathcal
F^{\bullet}[-n])\,t^{\alpha},
$$
where $\mathbf h^{\alpha}(\mathcal F^{\bullet}[-n]):=\dim_{\mathbb
Q} {\mathbb H}^{\alpha}(\mathcal F^{\bullet}[-n])$ denotes the
dimension of the hypercohomology.

\medskip
$(viii)$ Let $\mathbb G_k(\mathbb C^l)$ denote the Grassmann variety of $k$ planes in
$\mathbb C^l$ (compare with \cite[p. 328]{CGM}). Recall that
$$
\dim \mathbb G_k(\mathbb C^l)=k(l-k),
$$
and that the Poincar\'e polynomial
$$
Q_k^l:=Q_k^l(t):=\sum_{\alpha\in \mathbb Z}\dim H^{\alpha}(\mathbb G_k(\mathbb C^l))t^{\alpha}
$$
of $\mathbb G_k(\mathbb C^l)$ is equal to
\begin{equation}\label{grass}
Q_k^l=\frac{P_l}{P_kP_{l-k}},
\end{equation}
where, for every integer $\alpha\geq 0$, we set:
$$
P_{\alpha}:=P_{\alpha}(t):=h_0\cdot h_1\cdot\dots\cdot h_{\alpha-1}\quad{\text{and}}\quad
h_{\alpha}:=h_{\alpha}(t):=1+t^2+t^4+\dots+t^{2\alpha}
$$
(assume that $P_0=1$, and notice that $P_1=1$).

\bigskip
\section{The main results}

We are in position to state our main results. We keep the notations
stated before, together the assumptions $(a1)$, $(a2)$, and $(a3)$.

\begin{theorem}
\label{ExDecTh} In $D^b(X)$ we have a decomposition:
$$R\,\pi_*\mathbb Q_{\widetilde X}[n]\cong IC^{\bullet}_X\oplus  \mathcal F^{\bullet}.$$
\end{theorem}

\begin{corollary}
\label{ExIC}
$$
IH_X(t)=H_{\widetilde X}(t)-H_{\Delta}(t)g(t).
$$
\end{corollary}

\medskip
In order to prove our results, we need the following:

\begin{lemma} \label{lemmaabc} $(a)$ $\mathcal
F^{\bullet}$ is self-dual in $D^b(X)$.

\smallskip $(b)$ $\mathcal F^{\bullet}$ is a
direct summand of $R{\pi}_{*}\mathbb Q_{\widetilde X}[n]$ in
$D^b(X)$.

\smallskip
$(c)$ For every $x\in\Delta$ and $\alpha\geq -m$, one has
$$\mathcal H^\alpha\left(R{\pi}_{*}\mathbb Q_{\widetilde
X}[n]\right)_x\cong \mathcal H^\alpha\left(\mathcal
F^{\bullet}\right)_x\cong A^{\alpha+n}.$$

\smallskip
$(d)$ $f(t)=H_{\Delta}(t)g(t).$
\end{lemma}

\begin{proof}[Proof of the Lemma]
$(a)$ If we set $\beta=p-q-\alpha$, we have:
$$
\sum_{\alpha=0}^{p-q-1}A^{\alpha}\otimes_{\mathbb Q}
R{\imath}_*\mathbb Q_{\Delta}[n-2q-\alpha] =
\sum_{\beta=1}^{p-q}A^{p-q-\beta}\otimes_{\mathbb Q}
R{\imath}_*\mathbb Q_{\Delta}[m+\beta].
$$
On the other hand, setting $\beta=\alpha-(p-q)$, we have:
$$
\sum_{\alpha=p-q+1}^{2p-2q}C^{\alpha}\otimes_{\mathbb Q}
R{\imath}_*\mathbb Q_{\Delta}[n-2q-\alpha] =
\sum_{\beta=1}^{p-q}C^{p-q+\beta}\otimes_{\mathbb Q}
R{\imath}_*\mathbb Q_{\Delta}[m-\beta].
$$
By (\ref{PD}), we deduce:
$$
\mathcal F^{\bullet}\cong \left(A^{p-q}\otimes_{\mathbb
Q}R{\imath}_*\mathbb Q_{\Delta}[m]\right) \oplus
\left(\sum_{\beta=1}^{p-q}A^{p-q-\beta}\otimes_{\mathbb Q}
R{\imath}_*\left(\mathbb Q_{\Delta}[m+\beta]\oplus \mathbb
Q_{\Delta}[m-\beta]\right)\right).
$$
Taking into account that $\mathbb Q_{\Delta}[m]$ is self-dual in
$D^b(\Delta)$, previous formula shows $\mathcal F^{\bullet}$ as a
direct sum of self-dual complexes (compare with \cite[p. 69,
Proposition 3.3.7 (ii), and Remark 3.3.6 (i)]{Dimca2}).

\medskip
$(b)$ Consider the commutative diagram (\ref{D1}). The inclusion
$\widetilde\Delta\stackrel{\jmath}{\hookrightarrow} \widetilde X$
gives rise a pull-back morphism $\mathbb Q_{\widetilde X}\to
R{\jmath}{_*}\mathbb Q_{\widetilde\Delta}$ in $D^b(\widetilde X)$
\cite[p. 176]{Voisin}. On the other hand, since $\widetilde\Delta$
is smooth (Notations, $(iii)$, $(a2)$), the inclusion
$\widetilde\Delta\stackrel{\jmath}{\hookrightarrow} \widetilde X$
induces also a Gysin morphism $R{\jmath}{_*}\mathbb
Q_{\widetilde\Delta}\to \mathbb Q_{\widetilde X}[2q]$ \cite[p.
83]{FultonCF}. So, we have the following sequence of morphisms:
$$
R\jmath{_*}\mathbb Q_{\widetilde\Delta}\to \mathbb
Q_{\widetilde X}[2q]\to
R{\jmath}{_*}\mathbb Q_{\widetilde\Delta}[2q].
$$
Composing with $\pi$, and taking into account that diagram
(\ref{D1}) commutes, we deduce the sequence in $D^b(X)$:
\begin{equation}\label{firstd}
R(\imath\circ \rho){_*}\mathbb
Q_{\widetilde\Delta}\to R{\pi}_{*}\mathbb
Q_{\widetilde X}[2q]\to R(\imath\circ
\rho){_*}\mathbb Q_{\widetilde\Delta}[2q].
\end{equation}

\smallskip
The idea of the proof consists in using the Leray-Hirsch
decomposition (\ref{LH}), and  the self-intersection formula
\cite[p. 92]{FultonCF}, \cite{Jou}, \cite{Yama}, in order to
identify the image of $\mathcal F^{\bullet}[-n+2q]$ via the
composite morphism $R(\imath\circ \rho){_*}\mathbb
Q_{\widetilde\Delta}\to R(\imath\circ \rho){_*}\mathbb
Q_{\widetilde\Delta}[2q]$ given by (\ref{firstd}).

\smallskip
More precisely, consider the complex $\mathcal F^{\bullet}$
(Notations, $(v)$). We may write:
$$
\mathcal
F^{\bullet}[-n+2q]=\left(\sum_{\alpha=0}^{p-q}A^{\alpha}\otimes_{\mathbb
Q} R{\imath}_*\mathbb Q_{\Delta}[-\alpha]\right) \oplus
\left(\sum_{\alpha=p-q+1}^{2p-2q}C^{\alpha}\otimes_{\mathbb Q}
R{\imath}_*\mathbb Q_{\Delta}[-\alpha]\right).
$$
By the Leray-Hirsch Theorem (\ref{LH}), we have:
$$
\sum_{\alpha=0}^{2p}A^{\alpha}\otimes R{\imath}_{*}\mathbb
Q_{\Delta}[-\alpha]\stackrel{\theta}{\cong}R({\imath}\circ
\rho){_*}\mathbb Q_{\widetilde\Delta}.
$$
Since $C^{\alpha}\subseteq A^{\alpha}$ (Notations, $(iv)$), we
deduce a morphism:
$$
\mathcal F^{\bullet}[-n+2q]\rightarrow R({\imath}\circ
\rho){_*}\mathbb Q_{\widetilde\Delta},
$$
and by (\ref{firstd}) we get a sequence:
$$
\mathcal F^{\bullet}[-n+2q]\to R{\pi}_{*}\mathbb Q_{\widetilde
X}[2q] \to R({\imath}\circ \rho){_*}\mathbb
Q_{\widetilde\Delta}[2q]\stackrel{}{\cong}\left(\sum_{\alpha=0}^{2p}A^{\alpha}\otimes
R{\imath}_{*}\mathbb Q_{\Delta}[-\alpha]\right)[2q].
$$
By the self-intersection formula, and the assumption $(a3)$ (compare
also with Notations, $(iv)$), the composite  of these morphisms
sends $\mathcal F^{\bullet}[-n+2q]$ isomorphically onto a subcomplex
$\mathcal G^{\bullet}$ of
$$\left(\sum_{\alpha=0}^{2p}A^{\alpha}\otimes R{\imath}_{*}\mathbb
Q_{\Delta}[-\alpha]\right)[2q]=
\sum_{\alpha=-2q}^{2p-2q}A^{\alpha+2q}\otimes R{\imath}_{*}\mathbb
Q_{\Delta}[-\alpha],$$ which, up to change the cohomology extension
$\theta$, identifies with:
$$
\mathcal
G^{\bullet}\cong\left(\sum_{\alpha=0}^{p-q}B^{\alpha+2q}\otimes_{\mathbb
Q} R{\imath}_*\mathbb Q_{\Delta}[-\alpha]\right)\oplus
\left(\sum_{\alpha=p-q+1}^{2p-2q}A^{\alpha+2q}\otimes_{\mathbb Q}
R{\imath}_*\mathbb Q_{\Delta}[-\alpha]\right).
$$
It follows that the morphism $\mathcal F^{\bullet}[-n+2q]\to
R{\pi}_{*}\mathbb Q_{\widetilde X}[2q]$ has a section. Therefore
$\mathcal F^{\bullet}[-n+2q]$ is a direct summand of
$R{\pi}_{*}\mathbb Q_{\widetilde X}[2q]$, i.e.  $\mathcal
F^{\bullet}$ is a direct summand of $R{\pi}_{*}\mathbb Q_{\widetilde
X}[n]$.

\medskip $(c)$
By \cite[Theorem 2.3.26, (i), p. 41]{Dimca2}, it follows that
$$\mathcal
H^\alpha\left(R{\pi}_{*}\mathbb Q_{\widetilde X}[n]\right)_x\cong
A^{\alpha+n}$$ for every $x\in \Delta$ and every $\alpha\in\mathbb
Z$. On the other hand, for every $\alpha,\beta\in\mathbb Z$, one
has:
$$
\mathcal H^\alpha\left(R{\imath}_{*}\mathbb
Q_{\Delta}[n-2q-\beta]\right)_x\cong
\begin{cases}\mathbb Q \quad{\text{if $\alpha+n-2q-\beta=0$}}\\
0\quad{\text{otherwise.}}
\end{cases}
$$
It follows that, for every $x\in \Delta$ and every $\alpha\geq -m$,
one has:
$$
\mathcal H^\alpha\left(\mathcal F^{\bullet}\right)_x\cong
\begin{cases} A^{p-q} \quad{\text{if $\alpha=-m$}}\\
C^{\alpha+n-2q}\quad{\text{if $\alpha>-m$.}}
\end{cases}
$$
We are done because,  by (Notations, $(iv)$) and (\ref{PD}), in the
case $\alpha=-m$, we have: $A^{p-q}\cong C^{p-q}\cong
A^{p+q}=A^{-m+n}=A^{\alpha+n}$, and, in the case $\alpha>-m$, we
have $C^{\alpha+n-2q}\cong A^{\alpha+n}$.

\medskip $(d)$
First we
analyze the summand on the left  of $\mathcal
F^{\bullet}[-n]$. Set $$h^{\alpha}(\Delta)=\dim_{\mathbb Q} H^{\alpha}(\Delta).$$ We
have:
$$
\sum_{\alpha\in\mathbb Z}\mathbf
h^{\alpha}\left(\sum_{\beta=0}^{p-q}A^{\beta} \otimes_{\mathbb Q}
R{\imath}_*\mathbb Q_{\Delta}[-2q-\beta]\right)t^{\alpha}
$$
$$
= \sum_{\alpha\in\mathbb
Z}\left(\sum_{\beta=0}^{p-q}a^{\beta}\cdot
h^{\alpha-\beta-2q}(\Delta)\right)t^{\alpha}
$$
$$
= \sum_{\alpha\in\mathbb
Z}\left(\sum_{\beta=0}^{p-q}\left(a^{\beta}t^{\beta+2q}\right)\cdot\left(
h^{\alpha-\beta-2q}(\Delta)t^{\alpha-\beta-2q}\right)\right)
$$
$$
=\sum_{\beta=0}^{p-q}\left(\sum_{\alpha\in\mathbb
Z}\left(a^{\beta}t^{\beta+2q}\right)\cdot\left(
h^{\alpha-\beta-2q}(\Delta)t^{\alpha-\beta-2q}\right)\right)
$$
$$
=\left(\sum_{\beta=0}^{p-q}a^{\beta}t^{\beta+2q}\right)H_{\Delta}(t).
$$
As for the summand on the right of $\mathcal F^{\bullet}[-n]$, we
have:
$$
\sum_{\alpha\in\mathbb Z}\mathbf
h^{\alpha}\left(\sum_{\beta=p-q+1}^{2p-2q}A^{\beta+2q}
\otimes_{\mathbb Q} R{\imath}_*\mathbb
Q_{\Delta}[-2q-\beta]\right)t^{\alpha}
$$
$$=\sum_{\alpha\in\mathbb
Z}\left(\sum_{\beta=p-q+1}^{2p-2q}a^{\beta+2q}\cdot
h^{\alpha-\beta-2q}(\Delta)\right)t^{\alpha}
$$
$$
=\sum_{\alpha\in\mathbb
Z}\left(\sum_{\beta=p-q+1}^{2p-2q}\left(a^{\beta+2q}t^{\beta+2q}\right)\cdot\left(
h^{\alpha-\beta-2q}(\Delta)t^{\alpha-\beta-2q}\right)\right)
$$
$$
=\sum_{\beta=p-q+1}^{2p-2q}\left(\sum_{\alpha\in\mathbb
Z}\left(a^{\beta+2q}t^{\beta+2q}\right)\cdot\left(
h^{\alpha-\beta-2q}(\Delta)t^{\alpha-\beta-2q}\right)\right)
$$
$$
=
\left(\sum_{\beta=p-q+1}^{2p-2q}a^{\beta+2q}t^{\beta+2q}\right)H_{\Delta}(t).
$$
Putting together we get:
$$
f(t)=\left(\sum_{\beta=0}^{p-q}a^{\beta}t^{\beta+2q}+
\sum_{\beta=p-q+1}^{2p-2q}a^{\beta+2q}t^{\beta+2q}\right)H_{\Delta}(t).
$$
Now we notice that:
$$
\sum_{\beta=0}^{p-q}a^{\beta}t^{\beta+2q}=t^{2q}\left(\sum_{\beta=0}^{p-q}a^{\beta}t^{\beta}\right).
$$
On the other hand, when $p-q+1\leq \beta \leq 2p-2q$, by the
Poincar\'e Duality Theorem, we have:
$$
a^{\beta+2q}t^{\beta+2q}=\frac{a^{\beta+2q}}{t^{2p-2q-\beta}}t^{2p}=
\frac{a^{2p-2q-\beta}}{t^{2p-2q-\beta}}t^{2p},
$$
with $0\leq 2p-2q-\beta \leq p-q-1$. Therefore, we have:
$$
\sum_{\beta=p-q+1}^{2p-2q}a^{\beta+2q}t^{\beta+2q}
=t^{2p}\left(\sum_{\beta=0}^{p-q-1}\frac{a^{\beta}}{t^{\beta}}\right).
$$
It follows that
$$
f(t)=\left[
t^{2q}\left(\sum_{\beta=0}^{p-q}a^{\beta}t^{\beta}\right) +t^{2p}
\left(\sum_{\beta=0}^{p-q-1}\frac{a^{\beta}}{t^{\beta}}\right)\right]
H_{\Delta}(t)=g(t)H_{\Delta}(t).
$$
\end{proof}

\smallskip We are in position to prove Theorem \ref{ExDecTh} and
Corollary \ref{ExIC}.

\medskip
\begin{proof}[Proof of Theorem \ref{ExDecTh}]
By Lemma \ref{lemmaabc}, (b),  there exists a complex $\mathcal
K^{\bullet}$ such that
\begin{equation}\label{ddec}
R\,\pi_*\mathbb Q_{\widetilde X}[n]\cong \mathcal
K^{\bullet}\oplus\mathcal F^{\bullet}.
\end{equation}
Therefore, we only have to prove that:
$$
\mathcal K^{\bullet}\cong IC_X^{\bullet}.
$$
Observe that $\mathcal K^{\bullet}$ is self-dual, because, by
\cite[p. 69, Proposition 3.3.7]{Dimca2} and Lemma \ref{lemmaabc},
(a), so are both $R\,\pi_*\mathbb Q^{\bullet}_{\widetilde X}[n]$ and
$\mathcal F^{\bullet}$. Now set $U:=X\backslash {\Delta}$, and
denote by ${\jmath}_U:U\hookrightarrow X$ the inclusion. Since the
complex $\mathcal F^{\bullet}$ is supported on $\Delta$, by
(Notations, $(iii)$, $(a1)$), it follows that the restriction
$(\jmath_U)^{-1}\mathcal K^{\bullet}$ of $\mathcal K^{\bullet}$ to
$U$ is $\mathbb Q_U[n]$. Moreover, by (\ref{ddec}), we have
$\mathcal K^{\bullet}\in D^b_c(X)$ \cite{Dimca2}. Therefore,
$\mathcal K^{\bullet}$ is an extension of $\mathbb Q_U[n]$. Hence,
to prove that $\mathcal K^{\bullet}\cong IC^{\bullet}_X$, it
suffices to prove that $\mathcal K^{\bullet}\cong
(\jmath_U)_{!*}\mathbb Q_U[n]$, i.e. that $\mathcal K^{\bullet}$ is
the intermediary extension of $\mathbb Q_U[n]$ \cite[p.156 and
p.135]{Dimca2}. Taking into account that $\mathcal K^{\bullet}$ is
self-dual, this in turn reduces to prove that, for every
$x\in\Delta$ and every $\alpha\geq -m$, one has $\mathcal
H^\alpha(\mathcal K^{\bullet})_x=0$ \cite[Proposition 5.2.8., p.135,
and Remark 5.4.2., p. 156]{Dimca2}. This follows from (\ref{ddec}),
and Lemma \ref{lemmaabc}, (c).
\end{proof}

\medskip
\begin{proof}[Proof of Corollary \ref{ExIC}] It follows from
Theorem \ref{ExDecTh} and Lemma \ref{lemmaabc}, $(d)$, taking into
account that
$IH^\alpha(X)=\mathbb H^\alpha(IC_X^{\bullet}[-n])$.
\end{proof}

\bigskip
\section{Example: single condition Schubert varieties with two strata.}

Fix integers $i,j,k,l$ such that:
$$
0\leq i\leq j\leq l, \quad\quad 0\leq i\leq k\leq l, \quad\quad \min\{j,k\}=i+1.
$$
Let $F^j\subseteq \mathbb C^l$ denote a fixed $j$-dimensional
subspace, and let $\mathbb G_k(\mathbb C^l)$ denote the Grassmann
variety of $k$ planes in $\mathbb C^l$. Define
$$
\mathcal S:=\left\{V^k\in \mathbb G_k(\mathbb C^l): \dim V^k\cap
F^j\geq i\right\}.
$$
$\mathcal S$ is called a {\it single condition Schubert variety}
\cite[p. 328]{CGM}, and we say {\it with two strata} because
$\min\{j,k\}=i+1$ (see (\ref{flag}) below).

\smallskip
{\it Our aim is  to compute the Poincar\'e polynomial $IH_{\mathcal
S}(t)$ of the intersection cohomology of ${\mathcal S}$, using
Corollary \ref{ExIC} with $X={\mathcal S}$.}

\smallskip
To this purpose,  consider the map \cite[p. 328]{CGM}:
$$
\pi: \widetilde{\mathcal S}\to  \mathcal S,
$$
where
$$
\widetilde{\mathcal S}:=\left\{(W^i,V^k)\in \mathbb G_i(F^j)\times
\mathbb G_k(\mathbb C^l): W^i\subseteq V^k\right\},
\quad{\text{and}}\quad \pi(W^i,V^k)=V^k.
$$
The map $\pi$ is a resolution of singularities of $\mathcal S$. We
have:
$$
{\text{Sing}}(\mathcal S)=\left\{V^k\in \mathbb G_k(\mathbb C^l):
\dim V^k\cap F^j>i\right\}.
$$
Since $\min\{j,k\}=i+1$, it follows that:
\begin{equation}\label{dsmooth}
{\text{Sing}}(\mathcal S)\cong
\begin{cases}
\mathbb G_{k-j}(\mathbb C^{l-j})\quad{\text{if $i+1=j$}}\\
\mathbb G_k(\mathbb C^j)\quad\,\,\,\,\,  \quad{\text{if $i+1=k$}}.
\end{cases}
\end{equation}
Therefore, ${\text{Sing}}(\mathcal S)$ is nonsingular. Moreover,
$\pi$ induces an isomorphism
\begin{equation}\label{isom}
\pi^{-1}(\mathcal S\backslash{\text{Sing}}(\mathcal S))\cong
\mathcal S\backslash {\text{Sing}}(\mathcal S),
\end{equation}
and
$$
\pi^{-1}\left({\text{Sing}}(\mathcal S)\right)\to
{\text{Sing}}(\mathcal S)
$$
is a smooth fibration, with
\begin{equation}\label{Fiber}
\pi^{-1}(x)\cong \mathbb P^i
\end{equation}
for every $x\in {\text{Sing}}(\mathcal S)$. So, the flag
\begin{equation}\label{flag}
\mathcal S\supseteq {\text{Sing}}(\mathcal S)
\end{equation} is a
stratification of $\mathcal S$ adapted to $\pi$ \cite{CGM},
\cite{Williamson}. Observe that the natural projection:
$$
(W^i,V^k)\in \widetilde{\mathcal S}\to W^i\in \mathbb G_i(F^j)
$$
is a smooth fibration, with base space $\mathbb G_i(F^j)$ and fiber
$\mathbb G_{k-i}(\mathbb C^{l-i})$. Therefore, the Poincar\'e
polynomial $H_{\widetilde {\mathcal S}}(t)$ of the cohomology of
$\widetilde {\mathcal S}$ is (compare with (\ref{grass})):
\begin{equation}\label{des}
H_{\widetilde {\mathcal
S}}(t)=Q_i^jQ_{k-i}^{l-i}=\frac{P_j}{P_iP_{j-i}}\cdot
\frac{P_{l-i}}{P_{k-i}P_{l-k}}.
\end{equation}

\smallskip
The map $\pi$ is said a {\it small resolution} of $\mathcal S$ if
and only if, for every $x\in {\text{Sing}}(\mathcal S)$, one has
$$
\dim\pi^{-1}(x)<\frac{1}{2}\left(\dim \mathcal S-\dim
{\text{Sing}}(\mathcal S)\right).
$$
Since
$$
\dim \mathcal S=i(j-i)+(k-i)(l-k),
$$
and
$$
\dim \mathcal S-\dim{\text{Sing}}(\mathcal S)=2i+1+l-j-k,
$$
it follows that
$$
{\text{\it $\pi$ is a small resolution of $\mathcal S$ if and only
if}}\,\,\,\,\, l-j-k\geq 0.
$$
In this case, one knows that $IH_{\mathcal S}(t)$ is equal to the
Poincar\'e polynomial $H_{\widetilde {\mathcal S}}(t)$ of the
cohomology of $\widetilde {\mathcal S}$ \cite{CGM}, \cite{GMP2}:
$$
\pi \quad{\text{small}} \implies IH_{\mathcal S}(t)=H_{\widetilde
{\mathcal S}}(t).
$$
Hence, if $\pi$ is small, i.e. if $l-j-k\geq 0$, by (\ref{des}) we
get:
\begin{equation}\label{Cheeger}
IH_{\mathcal S}(t)=Q_i^jQ_{k-i}^{l-i}=\frac{P_j}{P_iP_{j-i}}\cdot
\frac{P_{l-i}}{P_{k-i}P_{l-k}}.
\end{equation}
This argument  appears in \cite[p. 329]{CGM} (see also \cite[p.
110-113]{Kirwan}). It applies only if $\pi$ is a small resolution,
bypassing the Decomposition Theorem.

\smallskip
When $\pi$ is non-small, we may apply our Corollary \ref{ExIC}. In
fact, if we set $ \Delta={\text{Sing}}(\mathcal S)$, then the map
$\pi: \widetilde{\mathcal S}\to  \mathcal S$ verifies all the
assumptions $(a1)$, $(a2)$, $(a3)$ stated in Notations, $(iii)$ (see
Lemma \ref{concl} below), and therefore, by Corollary \ref{ExIC}, we
get:
\begin{equation}\label{cors}
IH_{\mathcal S}(t)=H_{\widetilde {\mathcal S}}(t)-H_{\Delta}(t)g(t).
\end{equation}
In order to explicit this formula, we distinguish the
cases $i+1=j$ and $i+1=k$.

\medskip $\bullet$
In the case $i+1=j$,   comparing with  the invariants defined in
Notations, we have:

\medskip
1) $\Delta \cong \mathbb G_{k-j}(\mathbb C^{l-j})$ and, by
(\ref{Fiber}), $G\cong \mathbb P^{j-1}$;

2) $n=j-1+(k-j+1)(l-k)$;

3) $m=(l-k)(k-j)$;

4) $p=j-1$;

5) $q=l-k$;

6) $p-q=j+k-l-1$ and therefore
$$
\pi\quad {\text{is small}}\quad \iff \quad l-j-k\geq 0 \quad\iff
\quad p-q<0;
$$

7) $a^{\alpha}=\dim H^{\alpha}(\mathbb P^{j-1})$;

8) $g(t)=t^{2(l-k)}+t^{2(l-k+1)} +\dots+t^{2(j-1)}$.

\medskip
By (\ref{grass}), (\ref{des}), and (\ref{cors}), we deduce (recall
that $P_1=1$):
\begin{equation}\label{f1}
IH_{\mathcal
S}(t)=\frac{P_{j}}{P_{j-1}}\cdot\frac{P_{l-j+1}}{P_{k-j+1}P_{l-k}}-\left(t^{2(l-k)}+t^{2(l-k+1)}
+\dots+t^{2(j-1)}\right)\cdot\frac{P_{l-j}}{P_{k-j}P_{l-k}},
\end{equation}
where  $t^{2(l-k)}+t^{2(l-k+1)}+\dots+t^{2(j-1)}$ denotes the zero
polynomial when $j+k\leq l$ (compare with Notations, $(vii)$).
Hence, previous formula reduces to (\ref{Cheeger}) in the small
case.

\medskip $\bullet$
In the case $i+1=k$,   the invariants  are:

\medskip
1) $\Delta \cong \mathbb G_{k}(\mathbb C^{j})$ and, by
(\ref{Fiber}), $G\cong \mathbb P^{k-1}$;

2) $n=(k-1)(j-k+1)+l-k$;

3) $m=k(j-k)$;

4) $p=k-1$;

5) $q=l-j$;

6) $p-q=j+k-l-1$ and therefore
$$
\pi\quad {\text{is small}}\quad \iff \quad l-j-k\geq 0 \quad\iff
\quad p-q<0;
$$

7) $a^{\alpha}=\dim H^{\alpha}(\mathbb P^{k-1})$;

8) $g(t)=t^{2(l-j)}+t^{2(l-j+1)} +\dots+t^{2(k-1)}$.

\medskip
By (\ref{grass}), (\ref{des}), and (\ref{cors}), we deduce:
\begin{equation}\label{f2}
IH_{\mathcal S}(t)=
\frac{P_{j}}{P_{k-1}P_{j-k+1}}\cdot\frac{P_{l-k+1}}{P_{l-k}}-\left(t^{2(l-j)}+t^{2(l-j+1)}
+\dots+t^{2(k-1)}\right)\cdot\frac{P_{j}}{P_{k}P_{j-k}},
\end{equation}
where  $t^{2(l-j)}+t^{2(l-j+1)}+\dots+t^{2(k-1)}$ denotes the zero
polynomial when $j+k\leq l$, and previous formula reduces to
(\ref{Cheeger}) in the small case.

\medskip
\begin{lemma}\label{concl} The resolution of singularities
$\pi:\widetilde{\mathcal S}\to \mathcal S$, with $\Delta=
{\text{Sing}}(\mathcal S)$, verifies all the assumptions $(a1)$,
$(a2)$, $(a3)$ stated in Notations, $(iii)$.
\end{lemma}

\begin{proof} In view of the description of the map $\pi:\widetilde{\mathcal
S}\to \mathcal S$ given in (\ref{dsmooth}), (\ref{isom}), and
(\ref{Fiber}), we only have to verify the assumption $(a3)$.

First we examine the case $i+1=j$.

In this case, we have
$$
\widetilde {\mathcal S}= \left\{(W^{j-1},V^k)\in \mathbb
G_{j-1}(F^j)\times \mathbb G_k(\mathbb C^l): W^{j-1}\subseteq
V^k\right\},
$$
and
$$
\widetilde {\Delta}= \left\{(W^{j-1},V^k)\in \mathbb
G_{j-1}(F^j)\times \mathbb G_k(\mathbb C^l): F^j\subseteq
V^k\right\}.
$$
Let  $S_{j-1}$ denote the tautological bundle on $\mathbb
G_{j-1}(F^j)\cong \mathbb P^{j-1}$, and $S_{k}$ the tautological
bundle on $\mathbb G_{k}(\mathbb C^l)$. Let $S'_{j-1}$ and $S'_{k}$
denote the pull-back of $S_{j-1}$ and $S_{k}$ via the natural
projections $\widetilde{\Delta}\to \mathbb G_{j-1}(F^j)$ and
$\widetilde{\Delta}\to \mathbb G_{k}(\mathbb C^l)$. We have
identifications with Grassmann bundles \cite[p. 434,
B.5.7]{FultonIT}:
$$
\widetilde {\mathcal S}\cong \mathbb G_{k-j+1}(\mathbb C^l\slash
S_{j-1}) \quad {\text{and}}\quad\widetilde {\Delta}\cong \mathbb
G_{k-j}(\mathbb C^l\slash F^j),
$$
where $\mathbb C^l$ and $F^j$ denote the trivial vector bundles (in
this case, on $\mathbb G_{j-1}(F^j)$). The relative tangent bundles
are (compare with \cite[p. 435, B.5.8]{FultonIT}):
$$
\jmath^* T_{\widetilde {\mathcal S}\slash \mathbb G_{j-1}(F^j)}\cong
{\text{Hom}}(S'_k\slash S'_{j-1}, \mathbb C^l\slash S'_{k}),
$$
and
$$
T_{\widetilde {\Delta}\slash \mathbb G_{j-1}(F^j)}\cong
{\text{Hom}}(S'_k\slash F^j, \mathbb C^l\slash S'_{k}),
$$
where $\jmath$ denotes the inclusion $\widetilde
{\Delta}\hookrightarrow \widetilde {\mathcal S}$, and $\mathbb C^l$
and $F^j$ denote the trivial vector bundles  on $\widetilde
{\Delta}$.  Therefore, applying ${\text{Hom}}(\cdot\,,\mathbb
C^l\slash S'_{k})$ to the exact sequence
$$
0\to  F^j\slash S'_{j-1}\to S'_k\slash S'_{j-1}\to S'_k\slash F^j\to
0,
$$
we get the exact sequence:
$$
0\to T_{\widetilde {\Delta}\slash \mathbb G_{k-1}(F^j)}\to \jmath^*
T_{\widetilde {\mathcal S}\slash \mathbb G_{k-1}(F^j)}\to
{\text{Hom}}(F^j\slash S'_{j-1}, \mathbb C^l\slash S'_{k})\to 0.
$$
It enables us to identify the normal bundle $N$ of  $\widetilde
{\Delta}$ in $\widetilde {\mathcal S}$ \cite[p. 438,
B.7.2]{FultonIT}:
$$
N\cong {\text{Hom}}(F^j\slash S'_{j-1}, \mathbb C^l\slash S'_{k}).
$$
It follows that the restriction $N_{|G}$ of $N$ to the fiber $G\cong
\mathbb P^{j-1}$ of $\rho:\widetilde {\Delta}\hookrightarrow \Delta$
is:
$$
N_{|G}\cong \mathcal O_{G}(-1)\otimes \mathbb C^q.
$$
Hence:
$$
0\neq c=c_q(N_{|G})=(-h)^q\in H^{2q}(\mathbb P^{j-1})\cong
H^{2q}(G),
$$
where $h\in H^{2q}(\mathbb P^{j-1})$ denotes the hyperplane class.
This is enough to prove $(a3)$ because $G\cong \mathbb P^{j-1}$ is a
projective space.

\smallskip
Now we turn to the case $i+1=k$.

In this case, we have
$$
\widetilde {\mathcal S}= \left\{(W^{k-1},V^k)\in \mathbb
G_{k-1}(F^j)\times \mathbb G_k(\mathbb C^l): W^{k-1}\subseteq
V^k\right\},
$$
and
$$
\widetilde {\Delta}= \left\{(W^{k-1},V^k)\in \mathbb
G_{k-1}(F^j)\times \mathbb G_k(F^j): W^{k-1}\subseteq V^k\right\}.
$$
Let $S_{k-1}$ denote the tautological bundle on $\mathbb
G_{k-1}(F^j)$, and $S_{k}$ the tautological bundle on $\mathbb
G_{k}(F^j)$. Let $S'_{k-1}$ and $S'_{k}$ denote the pull-back of
$S_{k-1}$ and $S_{k}$ via the natural projection
$\widetilde{\Delta}\to \mathbb G_{k-1}(F^j)$. We have
identifications with projective bundles:
$$
\widetilde {\mathcal S}\cong \mathbb P(\mathbb C^l\slash S_{k-1}),
\quad {\text{and}}\quad \widetilde {\Delta}\cong \mathbb P(F^j\slash
S_{k-1}).
$$
The relative tangent bundles are \cite[p. 435, B.5.8]{FultonIT}:
$$
\jmath^* T_{\widetilde {\mathcal S}\slash \mathbb G_{k-1}(F^j)}\cong
{\text{Hom}}(S'_k\slash S'_{k-1}, \mathbb C^l\slash S'_{k}),
$$
and
$$
T_{\widetilde {\Delta}\slash \mathbb G_{k-1}(F^j)}\cong
{\text{Hom}}(S'_k\slash S'_{k-1}, F^j\slash S'_{k}),
$$
where $\jmath$ denotes the inclusion $\widetilde
{\Delta}\hookrightarrow \widetilde {\mathcal S}$. Therefore,
applying ${\text{Hom}}(S'_k\slash S'_{k-1}, \,\cdot)$ to the exact
sequence
$$
0\to F^j\slash S'_{k}\to \mathbb C^l\slash S'_{k}\to \mathbb C^q\to
0,
$$
we get the exact sequence:
$$
0\to T_{\widetilde {\Delta}\slash \mathbb G_{k-1}(F^j)}\to \jmath^*
T_{\widetilde {\mathcal S}\slash \mathbb G_{k-1}(F^j)}\to
{\text{Hom}}(S'_k\slash S'_{k-1}, \mathbb C^q)\to 0.
$$
Hence, the normal bundle $N$ of  $\widetilde {\Delta}$ in
$\widetilde {\mathcal S}$ \cite[p. 438, B.7.2]{FultonIT} is:
$$
N\cong {\text{Hom}}(S'_k\slash S'_{k-1}, \mathbb C^q).
$$
It follows that the restriction $N_{|G}$ of $N$ to the fiber $G\cong
\mathbb P^{k-1}$ of $\rho:\widetilde {\Delta}\hookrightarrow \Delta$
is:
$$
N_{|G}\cong \mathcal O_{G}(-1)\otimes \mathbb C^q.
$$
Hence:
$$
0\neq c=c_q(N_{|G})=(-h)^q\in H^{2q}(\mathbb P^{k-1})\cong
H^{2q}(G),
$$
where $h\in H^{2q}(\mathbb P^{k-1})$ denotes the hyperplane class.
\end{proof}

\medskip
\begin{remark}\label{concl2} $(i)$ Another resolution of $\mathcal S$ is
given by
$$
\pi_1: (V^k,U^{k+j-i})\in\widetilde{\mathcal S_1}\to  V^k\in\mathcal
S,
$$
where
$$
\widetilde{\mathcal S_1}:=\left\{(V^k,U^{k+j-i})\in \mathbb
G_k(\mathbb C^l)\times \mathbb G_{k+j-i}(\mathbb C^l):
V^k+F^j\subseteq U^{k+j-i}\right\}.
$$
A similar argument as before shows that
$$
{\text{\it $\pi_1$ is a small resolution of $\mathcal S$ if and only
if}}\,\,\,\,\, l-j-k\leq 0,
$$
and,  in this case, we have:
\begin{equation}\label{f3}
IH_{\mathcal S}(t)= H_{\widetilde {\mathcal
S_1}}(t)=Q_{k-i}^{l-j}Q_{k}^{k+j-i}=\frac{P_{l-j}}{P_{k-i}P_{l-j-k+i}}\cdot
\frac{P_{k+j-i}}{P_{k}P_{j-i}}. \end{equation} This  is another way
to compute $IH_{\mathcal S}(t)$ when $\pi$ is non-small, relying on
the same argument as in  \cite{CGM}.

\smallskip $(ii)$ Comparing (\ref{f1}), (\ref{f2}) and (\ref{f3}),
in the case $l\leq j+k$ we obtain the following polynomial
identities, that one may easily verify with a direct computation:

\medskip
if $i+1=j$ then:
$$
\frac{P_{l-j}}{P_{k-j+1}P_{l-k-1}}\cdot \frac{P_{k+1}}{P_{k}}
\quad\quad\quad\quad\quad\quad\quad\quad\quad\quad\quad\quad\quad
$$
$$
=\frac{P_{j}}{P_{j-1}}\cdot\frac{P_{l-j+1}}{P_{k-j+1}P_{l-k}}-\left(t^{2(l-k)}+t^{2(l-k+1)}
+\dots+t^{2(j-1)}\right)\cdot\frac{P_{l-j}}{P_{k-j}P_{l-k}};
$$

\medskip
if $i+1=k$ then:
$$
\frac{P_{l-j}}{P_{1}P_{l-j-1}}\cdot \frac{P_{j+1}}{P_{k}P_{j-k+1}}
\quad\quad\quad\quad\quad\quad\quad\quad\quad\quad\quad\quad\quad
$$
$$
=\frac{P_{j}}{P_{k-1}P_{j-k+1}}\cdot\frac{P_{l-k+1}}{P_{l-k}}-\left(t^{2(l-j)}+t^{2(l-j+1)}
+\dots+t^{2(k-1)}\right)\cdot\frac{P_{j}}{P_{k}P_{j-k}}.
$$
\end{remark}

\bigskip
\section{Example: hypersurfaces of $\mathbb P^5$ with one-dimensional singular locus}

\bigskip
Fix a smooth threefold  $T\subset \mathbb P^5$, complete
intersection, with equations $t_1=t_2=0$. Let $X\subset \mathbb P^5$
be a general hypersurface containing $T$, with equation
$t_1t_3-t_2t_4=0$.  By Bertini's theorem, the singular locus of $X$
is contained in $T$. Actually, since $T$ is smooth,
${\text{Sing}}(X)$ is equal to the smooth complete intersection
curve $\Delta$, defined by $t_1=t_2=t_3=t_4=0$. Set:
$$
d_i:=\deg t_i,\quad x:=\deg X=d_1+d_3=d_2+d_4,\quad
\delta:=\deg\Delta=d_1d_2d_3d_4.
$$
Observe that:
$$
\mathcal O_{\Delta}(K_{\Delta})\cong \mathcal O_{\Delta}(2x-6),\quad
2g-2=(2x-6)\delta,
$$
where $K_{\Delta}$ denotes the canonical divisor of $\Delta$, and
$g$ the genus.

\smallskip
Let $\sigma:\mathbb P\to \mathbb P^5$ be the blowing-up of $\mathbb
P^5$ along $\Delta$. Let $E\cong \Delta\times \mathbb P^3\subset
\mathbb P$ be the exceptional divisor. Let $\widetilde X\subset
\mathbb P$ be the strict transform of $X$, which is the blowing-up
of $X$ along $\Delta$. The restriction of $\sigma$ to $\widetilde
X$:
$$
\pi:\widetilde X\to X
$$
is a resolution of singularities of $X$. The exceptional divisor
$\widetilde \Delta$ of $\widetilde X$ is:
$$
\widetilde \Delta\cong \Delta \times G,
$$
where $G$ is the smooth quadric surface in $\mathbb P^3$. The
resolution $\pi$ verifies all the assumptions $(a1)$, $(a2)$,
$(a3)$, and therefore, by Corollary \ref{ExIC}, we get:
$$
IH_{X}(t)=H_{\widetilde {X}}(t)-H_{\Delta}(t)g(t).
$$
In order to explicit this formula, first we notice that, in this
example, the invariants are: $n=4$, $m=1$, $p=2$, $q=1$. So, we
have:
$$
H_{\Delta}(t)=1+2gt+t^2, \quad g(t)=t^2+t^4.
$$
It remains to compute $H_{\widetilde {X}}(t)$, i.e. the Betti
numbers $b_i(\widetilde X)$ of $\widetilde X$.

\smallskip
To this purpose, we recall some properties of $\mathbb P$, which we
will use in the sequel. We refer to \cite[p. 605]{GH}, \cite[p. 592,
Lemma 1.4 and Proof of Theorem 1.2]{BEL}, and \cite[p. 67, Example
3.3.4]{FultonIT} for more details. Set $\mathcal O_{\mathbb
P}(H)=\sigma^*\mathcal O_{\mathbb P^5}(1)$. We have:
\begin{equation}\label{adjj}
\mathcal O_{\mathbb P}(K_{\mathbb P})\cong \mathcal O_{\mathbb
P}(-6H+3E), \quad \mathcal O_{\mathbb P}(\widetilde X)\cong \mathcal
O_{\mathbb P}(xH-2E),
\end{equation}
and so
\begin{equation}\label{adj}
\mathcal O_{\mathbb P}(K_{\mathbb P}+\widetilde X)\cong \mathcal
O_{\mathbb P}((x-6)H+E), \quad \mathcal O_{\widetilde X
}(K_{\widetilde X})\cong \mathcal O_{\mathbb P}((x-6)H+E)\otimes
\mathcal O_{\widetilde X}.
\end{equation}

We also have:
\begin{equation}\label{int}
H^5=1, \quad H^4E=H^3E^2=H^2E^3=0,\quad HE^4=-\delta,
\end{equation}
$$
E^5=-c_1(N_{\Delta,\mathbb P^5})=2-2g-6\delta
$$
($N_{\Delta,\mathbb P^5}$ denotes the normal bundle of $\Delta$ in
$\mathbb P^5$).

Moreover:
\begin{equation}\label{BELuno}
H^{\alpha}(\mathbb P, \sigma^*M\otimes \mathcal O_{\mathbb
P}(iE))\cong H^{\alpha}(\mathbb P^5, M)
\end{equation}
for every vector bundle $M$ on $\mathbb P^5$,  every $\alpha$, and
every $0\leq i\leq 3$, and
\begin{equation}\label{BELdue}
H^{\alpha}(\mathbb P, \mathcal O_{\mathbb P}(iH-E))\cong
H^{\alpha}(\mathbb P^5, \mathcal I_{\Delta,\mathbb P^5}(i)),
\end{equation}
for every $\alpha$ and every $i$ ($\mathcal I_{\Delta,\mathbb P^5}$
denotes the ideal sheaf of $\Delta$ in $\mathbb P^5$).

\smallskip
We are in position to compute the Betti numbers of $\widetilde X$.

\begin{lemma} \label{Betti}
$$
b_1(\widetilde X)=0, \quad b_2(\widetilde X)=3, \quad b_3(\widetilde
X)=4g,
$$
$$
b_4(\widetilde X)=(x-2)(x^2-3x+3)(x^2-x+1)-(g-1)+3(2-\delta).
$$
\end{lemma}

\begin{corollary}\label{new}
$$
IH_{X}(t)=1+2t^2+2gt^3
\quad\quad\quad\quad\quad\quad\quad\quad\quad\quad\quad\quad
$$
$$
+[(x-2)(x^2-3x+3)(x^2-x+1)-(g-1)+(4-3\delta)]t^4+2gt^5+2t^6+t^8.
$$
\end{corollary}

\begin{proof}[Proof of the Lemma \ref{Betti}] For every $\alpha\in\mathbb Z$, consider the following natural commutative diagram:
$$
\begin{array}{ccccccc}
H_{\alpha+1}(\widetilde X, \widetilde \Delta)&\stackrel
{}{\longrightarrow}&H_{\alpha}(\widetilde
\Delta)&\stackrel{}{\longrightarrow}
&H_{\alpha}(\widetilde X)&\stackrel {}{\longrightarrow}&H_{\alpha}(\widetilde X, \widetilde \Delta)\\
\stackrel {}{}\Vert & &\stackrel {}{}\downarrow & & \stackrel
{}{}\downarrow& &
\stackrel {}{}\Vert \\
H_{\alpha+1}(X,\Delta)&\stackrel
{}{\longrightarrow}&H_{\alpha}(\Delta)&\stackrel{}{\longrightarrow}
&H_{\alpha}(X)&\stackrel {}{\longrightarrow}&H_{\alpha}(X,\Delta),
\end{array}
$$
where the horizontal rows are the homology exact sequences of the
couple, and the vertical maps are induced by $\pi$. As for the
isomorphism $H_{*}(\widetilde X, \widetilde \Delta)\cong H_{*}(X,
\Delta)$, see \cite[p. 23]{Lamotke}. By the Lefschetz Hyperplane
Theorem, we know that
$$
b_1(X)=b_3(X)=0,\quad b_2(X)=1.
$$
Combining with the K\"unneth formula for $\widetilde \Delta\cong
\Delta\times G$, by a simple diagram chase we deduce:

$\bullet$ $b_1(\widetilde X)=0$;

$\bullet$ the push-forward $H_2(\widetilde \Delta)\to H_2(\widetilde
X)$ is an isomorphism, therefore $b_2(\widetilde X)=b_2(\widetilde
\Delta)=3$;

$\bullet$ the push-forward $H_3(\widetilde \Delta)\to H_3(\widetilde
X)$ is onto.

\medskip
In particular, the pull-back $H^3(\widetilde X)\to H^3(\widetilde
\Delta)$ is injective. Since $$H^3(\widetilde\Delta)\cong
H^1(\Delta)\otimes H^2(G)\cong H^{1,2}(\widetilde\Delta)\oplus
H^{2,1}(\widetilde\Delta)\cong \mathbb C^{4g},$$ it follows that
$H^3(\widetilde X)=H^{1,2}(\widetilde X)\oplus H^{2,1}(\widetilde
X)$. Hence, in order to prove that $b_3(\widetilde X)=4g$, it
suffices to prove that
\begin{equation}\label{red}
h^{1,2}(\widetilde X)\geq 2g.
\end{equation}
To this aim, let $$N_{\widetilde X,\mathbb P}\cong \mathcal
O_{\mathbb P}(\widetilde X)\otimes \mathcal O_{\widetilde X}$$ be
the normal bundle of $\widetilde X$ in $\mathbb P$. From the natural
exact sequence
$$
0\to N_{\widetilde X,\mathbb P}^{\vee}\to \Omega_{\mathbb P}\otimes
\mathcal O_{\widetilde X}\to \Omega_{\widetilde X}\to 0,
$$
we get the following exact sequence:
\begin{equation}\label{esuno}
H^1(\widetilde X, \Omega_{\widetilde X})\to H^2(\widetilde X,
N_{\widetilde X,\mathbb P}^{\vee})\to H^2(\widetilde X,
\Omega_{\mathbb P}\otimes \mathcal O_{\widetilde X})\to
H^2(\widetilde X, \Omega_{\widetilde X}).
\end{equation}
In order to identify the first map $H^1(\widetilde X,
\Omega_{\widetilde X})\to H^2(\widetilde X, N_{\widetilde X,\mathbb
P}^{\vee})$, first notice that
\begin{equation}\label{isouno}
H^1(\widetilde X, \Omega_{\widetilde X})\cong H^2(\widetilde \Delta)
\end{equation}
because $$H^2(\widetilde X)\cong H^2(\widetilde \Delta)\cong
(H^0(\Delta)\otimes H^2(G))\oplus (H^2(\Delta)\otimes H^0(G))\cong
H^{1,1}(\widetilde \Delta).$$ On the other hand, by the Serre
Duality Theorem and (\ref{adj}), we have:
$$
H^2(\widetilde X, N_{\widetilde X,\mathbb P}^{\vee})\cong
H^2(\widetilde X, \mathcal O_{\mathbb P}((2x-6)H-E)\otimes \mathcal
O_{\widetilde X})^{\vee}.
$$
Tensoring the exact sequence
$$
0\to \mathcal O_{\mathbb P}(-\widetilde X)\to \mathcal O_{\mathbb
P}\to \mathcal O_{\widetilde X}\to 0
$$
with $\mathcal O_{\mathbb P}((2x-6)H-E)$, we get the exact sequence
$$
H^2(\mathbb P, \mathcal O_{\mathbb P}((x-6)H+E))\to H^2(\mathbb P,
\mathcal O_{\mathbb P}((2x-6)H-E))\to
$$
$$
\to H^2(\widetilde X, \mathcal O_{\mathbb P}((2x-6)H-E)\otimes
\mathcal O_{\widetilde X})\to H^3(\mathbb P, \mathcal O_{\mathbb
P}((x-6)H+E)).
$$
Now by (\ref{BELuno}) and (\ref{BELdue}) we have:
$$
H^2(\mathbb P, \mathcal O_{\mathbb P}((x-6)H+E))\cong H^2(\mathbb
P^5, \mathcal O_{\mathbb P^5}(x-6))=0,
$$
$$
H^3(\mathbb P, \mathcal O_{\mathbb P}((x-6)H+E))\cong H^3(\mathbb
P^5, \mathcal O_{\mathbb P^5}(x-6))=0,
$$
and
$$
H^2(\mathbb P, \mathcal O_{\mathbb P}((2x-6)H-E))\cong H^2(\mathbb
P^5, \mathcal I_{\Delta,\mathbb P^5}(2x-6))\cong
$$
$$
\cong H^1(\Delta, \mathcal O_{\Delta}(2x-6))\cong H^0(\Delta,
\mathcal O_{\Delta})^{\vee},
$$
because $\mathcal O_{\Delta}(2x-6)\cong \mathcal
O_{\Delta}(K_{\Delta})$. Summing up, we get
\begin{equation}\label{isodue}
H^2(\widetilde X, N_{\widetilde X,\mathbb P}^{\vee})\cong
 H^0(\Delta,
\mathcal O_{\Delta})\cong H^0(\Delta)\cong \mathbb C.
\end{equation}
By (\ref{isouno}) and (\ref{isodue}), it follows that  the map
$H^1(\widetilde X, \Omega_{\widetilde X})\to H^2(\widetilde X,
N_{\widetilde X,\mathbb P}^{\vee})$ identifies with the surjective
projection $H^2(\widetilde \Delta)\to H^0(\Delta)$ given by the
K\"unneth formula. By (\ref{esuno}), it follows an injective map
$$
0\to H^2(\widetilde X, \Omega_{\mathbb P}\otimes \mathcal
O_{\widetilde X})\to H^2(\widetilde X, \Omega_{\widetilde X}).
$$
Hence, by (\ref{red}), to prove that $b_3(\widetilde X)=4g$, it
suffices to prove that
\begin{equation}\label{reduno}
\dim H^2(\widetilde X, \Omega_{\mathbb P}\otimes \mathcal
O_{\widetilde X})\geq 2g. \end{equation} To this aim, consider again
the exact sequence
$$
0\to \mathcal O_{\mathbb P}(-\widetilde X)\to \mathcal O_{\mathbb
P}\to \mathcal O_{\widetilde X}\to 0.
$$
Tensoring with  $\Omega_{\mathbb P}$, and taking the cohomology, we
get the exact sequence:
$$
H^2(\mathbb P, \Omega_{\mathbb P}\otimes \mathcal O_{\mathbb P
}(-\widetilde X))\to H^2(\mathbb P, \Omega_{\mathbb P})\to
H^2(\widetilde X, \Omega_{\mathbb P}\otimes \mathcal O_{\widetilde
X}).
$$
Since $\dim H^2(\mathbb P, \Omega_{\mathbb P})=2g$ \cite[p. 180,
Theorem 7.31]{Voisin}, to prove (\ref{reduno}) (hence (\ref{red})),
it is enough to prove that
$$
H^2(\mathbb P, \Omega_{\mathbb P}\otimes \mathcal O_{\mathbb P
}(-\widetilde X))=0,
$$
i.e., by the Serre Duality Theorem, that
\begin{equation}\label{reddue}
H^3(\mathbb P, \mathcal T_{\mathbb P}\otimes \mathcal O_{\mathbb P
}(K_{\mathbb P}+\widetilde X))=0. \end{equation} Consider the exact
sequence \cite[p. 299]{FultonIT}
$$
0\to \mathcal T_{\mathbb P}\to \sigma^*\mathcal T_{\mathbb P^5}\to
j_*(F)\to 0,
$$
where $j: E\to \mathbb P$  denotes the inclusion, and $F$ the
universal quotient bundle on $E$. Tensoring with $\mathcal
O_{\mathbb P}(K_{\mathbb P}+\widetilde X)$, and taking the
cohomology, we get the exact sequence:
$$
H^2(\mathbb P, \sigma^*\mathcal T_{\mathbb P^5}\otimes \mathcal
O_{\mathbb P}(K_{\mathbb P}+\widetilde X))\to H^2(\mathbb P,
j_*(F)\otimes \mathcal O_{\mathbb P}(K_{\mathbb P}+\widetilde X))\to
$$
$$
\to H^3(\mathbb P, \mathcal T_{\mathbb P}\otimes \mathcal O_{\mathbb
P }(K_{\mathbb P}+\widetilde X))\to H^3(\mathbb P, \sigma^*\mathcal
T_{\mathbb P^5}\otimes \mathcal O_{\mathbb P}(K_{\mathbb
P}+\widetilde X)).
$$
By (\ref{adj}) and (\ref{BELuno}) we have:
$$
H^2(\mathbb P, \sigma^*\mathcal T_{\mathbb P^5}\otimes \mathcal
O_{\mathbb P}(K_{\mathbb P}+\widetilde X))\cong H^2(\mathbb P^5,
\mathcal T_{\mathbb P^5}\otimes \mathcal O_{\mathbb P^5}(x-6))=0,
$$
$$
H^3(\mathbb P, \sigma^*\mathcal T_{\mathbb P^5}\otimes \mathcal
O_{\mathbb P}(K_{\mathbb P}+\widetilde X))\cong H^3(\mathbb P^5,
\mathcal T_{\mathbb P^5}\otimes \mathcal O_{\mathbb P^5}(x-6))=0,
$$
and, by the projection formula, we have
$$
H^2(\mathbb P, j_*(F)\otimes \mathcal O_{\mathbb P}(K_{\mathbb
P}+\widetilde X)) \cong H^2(E, F\otimes j^*(\mathcal O_{\mathbb
P}(K_{\mathbb P}+\widetilde X))).
$$
It follows that
$$
H^3(\mathbb P, \mathcal T_{\mathbb P}\otimes \mathcal O_{\mathbb P
}(K_{\mathbb P}+\widetilde X))\cong  H^2(E, F\otimes j^*(\mathcal
O_{\mathbb P}(K_{\mathbb P}+\widetilde X))).
$$
So, in order to prove (\ref{reddue}), it suffices to prove that
\begin{equation}\label{redtre}
 H^2(E, F\otimes j^*(\mathcal O_{\mathbb P}(K_{\mathbb P}+\widetilde X)))=0.
\end{equation}
Consider the exact sequence \cite[loc. cit.]{FultonIT}
$$
0\to N_{E,\mathbb P}\to \tau^* N_{\Delta,\mathbb P^5}\to F\to 0,
$$
where $N_{E,\mathbb P}\cong \mathcal O_E\otimes \mathcal O_{\mathbb
P}(E)$ is the normal bundle of $E$ in $\mathbb P$, $\tau:E\to
\Delta$ is the natural projection, and $N_{\Delta,\mathbb P^5}\cong
\mathcal \oplus_{i=1}^4\mathcal O_{\Delta}(d_i)$ is the normal
bundle of $\Delta$ in $\mathbb P^5$. Tensoring with $j^*(\mathcal
O_{\mathbb P}(K_{\mathbb P}+\widetilde X))$, and taking into account
(\ref{adj}), we deduce that the proof of the  vanishing
(\ref{redtre}) (hence the proof of (\ref{red})) amounts to show
that:
\begin{equation}\label{ultimo}
H^2(E, j^*(\mathcal O_{\mathbb P}((d_i+x-6)H+E)))= H^3(E,
j^*(\mathcal O_{\mathbb P}((x-6)H+2E)))=0,
\end{equation}
for every $i=1,2,3,4$. Consider the exact sequence:
$$
0\to \mathcal O_{\mathbb P}(-E)\to \mathcal O_{\mathbb P}\to
\mathcal O_{E}\to 0.
$$
Tensoring with $\mathcal O_{\mathbb P}((d_i+x-6)H+E)$, we have the
exact sequence:
$$
H^2(\mathbb P, \mathcal O_{\mathbb P}((d_i+x-6)H+E))\to H^2(E,
j^*(\mathcal O_{\mathbb P}((d_i+x-6)H+E)))\to
$$
$$
\to H^3(\mathbb P, \mathcal O_{\mathbb P}((d_i+x-6)H)),
$$
and tensoring with $\mathcal O_{\mathbb P}((x-6)H+2E)$, we have the
exact sequence:
$$
H^3(\mathbb P, \mathcal O_{\mathbb P}((x-6)H+2E))\to H^3(E,
j^*(\mathcal O_{\mathbb P}((x-6)H+2E))) \to
$$
$$\to
H^4(\mathbb P, \mathcal O_{\mathbb P}((x-6)H+E)).
$$
By (\ref{BELuno}) we have:
$$
H^2(\mathbb P, \mathcal O_{\mathbb P}((d_i+x-6)H+E))\cong
H^2(\mathbb P^5, \mathcal O_{\mathbb P^5}(d_i+x-6))=0,
$$
$$
H^3(\mathbb P, \mathcal O_{\mathbb P}((d_i+x-6)H))\cong H^3(\mathbb
P^5, \mathcal O_{\mathbb P^5}(d_i+x-6))=0,
$$
$$
H^3(\mathbb P, \mathcal O_{\mathbb P}((x-6)H+2E))\cong H^3(\mathbb
P^5, \mathcal O_{\mathbb P^5}(x-6))=0,
$$
$$
H^4(\mathbb P, \mathcal O_{\mathbb P}((x-6)H+E))\cong H^4(\mathbb
P^5, \mathcal O_{\mathbb P^5}(x-6))=0.
$$
This proves the vanishing (\ref{ultimo}), and concludes the proof of
the equality $b_3(\widetilde X)=4g$.

\smallskip
Now we turn to $b_4(\widetilde X)$.

By the Gauss-Bonnet Formula \cite[p. 416]{GH}, we know that
$$
c_4(\mathcal T_{\widetilde X})=\chi_{\text{top}}(\widetilde X).
$$
Therefore, by the previous computations of $b_i(\widetilde X)$,
$i=1,2,3$, we have:
$$
b_4(\widetilde X)=c_4(\mathcal T_{\widetilde X})+4(2g-2).
$$
Hence, the computation of $b_4(\widetilde X)$ amounts to that of
$c_4(\mathcal T_{\widetilde X})$. By the exact sequence:
$$
0\to \mathcal T_{\widetilde X}\to \mathcal T_{\mathbb P}\otimes
\mathcal O_{\widetilde X}\to N_{\widetilde X,\mathbb P}\to 0,
$$
we get
\begin{equation}\label{ultimaf}
c_4(\mathcal T_{\widetilde X})=\widetilde X\cdot c_4(\mathcal
T_{\mathbb P})-\widetilde X^2\cdot c_3(\mathcal T_{\mathbb P}) +
\widetilde X^3\cdot c_2(\mathcal T_{\mathbb P})-\widetilde X^4\cdot
c_1(\mathcal T_{\mathbb P})+\widetilde X^5.
\end{equation}
On the other hand, using \cite[p. 300, Example 15.4.2]{FultonIT}, we
find:
$$
c_1(\mathcal T_{\mathbb P})=6H-3E,\quad c_2(\mathcal T_{\mathbb
P})=15H^2+2(x-9)HE+2E^2,
$$
$$
c_3(\mathcal T_{\mathbb P})=20H^3+8x(x-3)H^2E+4(3-x)HE^2+2E^3,
$$
$$
c_4(\mathcal T_{\mathbb P})=15H^4+12HE^3-3E^4.
$$
Inserting previous data into (\ref{ultimaf}), and taking into
account (\ref{adjj}) and (\ref{int}), we get:
$$
c_4(\mathcal T_{\widetilde
X})=(x-2)(x^2-3x+3)(x^2-x+1)-9(g-1)+3(2-\delta).
$$
\end{proof}

\end{document}